\documentclass[12pt]{amsart}
\pdfoutput=1
\usepackage[utf8]{inputenc}
\usepackage{MnSymbol}
\usepackage{accents}
\usepackage{tikz-cd}

\usepackage{bm}
\usepackage[textsize=footnotesize,shadow]{todonotes}
\usepackage{fancyhdr}
\usepackage[colorlinks=true,unicode=true]{hyperref}
\usepackage{mathtools}
\usepackage{lpic}

\input{style.tex}
\input{defs.tex}

\def\draft{\fbox{\tiny draft: \today}}
\let\draft=\relax
\setcolors[rmcolor=cyan,
           commentcolor=purple,
           inscolor=blue,
           commentsymbol=$\boxed{\mathbb{S}}$,
           rmsymbol=$\boxed{\mathbb{S}_{rm}}$,
           inssymbol=$\boxed{\mathbb{S}_{ins}}$]
          {s}
\setcolors[rmcolor=green,
           commentcolor=red,
           inscolor=brown,
           commentsymbol=$\boxed{\mathbb{J}}$,
           rmsymbol=$\boxed{\mathbb{J}_{rm}}$,
           inssymbol=$\boxed{\mathbb{J}_{ins}}$]
          {j}

\title{Tropical diagrams of probability spaces}
\author[RM]{R. Matveev}
\author[JWP]{J. W. Portegies}
\begin{document}
\thispagestyle{fancy} 

\begin{abstract}
  After endowing the space of diagrams of probability spaces with an
  entropy distance, we study its large-scale geometry by identifying
  the asymptotic cone as a closed convex cone in a Banach space. We
  call this cone the \emph{tropical cone}, and its elements
  \emph{tropical diagrams of probability spaces}. Given that the
  tropical cone has a rich structure, while tropical diagrams are
  rather flexible objects, we expect the theory of tropical diagrams
  to be useful for information optimization problems in information
  theory and artificial intelligence. In a companion article, we give
  a first application to derive a statement about the entropic cone.
\end{abstract}
\maketitle

\son
\jon

\section{Introduction}
With \cite{Matveev-Asymptotic-2018} we started a research program
aiming for a systematic approach to a class of information
optimization problems in information theory and artificial
intelligence.  A prototypical example of such a problem, still wide
open, is the characterization of the entropic cone, the closure of all
vectors in $\Rbb^{2^N-1}$, which are entropically
representable. Other information optimization problems arise for
instance in causal inference \cite{Steudel-Information-2015},
artificial intelligence \cite{Dijk-Informational-2013}, information
decomposition \cite{Bertschinger-Quantifying-2014},
robotics \cite{Ay-Predictive-2008},
neuroscience \cite{Friston-Free-2009} and in variational
autoencoders \cite{Kingma-Auto-2013}.

The global strategy of our program is roughly based on the following
way of thinking.  The entropic cone is clearly a very complicated
object: it is known that it is not
polyhedral \cite{Matus-Infinitely-2007}.  Yet, perhaps, much of its
complexity may be explained by it being a projection of another,
simpler, higher-dimensional object.

The purpose of this article is to construct such a higher-dimensional
(infinite-dimensional, in fact) object, which we call
the \term{tropical cone}, and to derive some of its basic properties.
In \cite{Matveev-Tropical-Entropic-2019} we apply the theory to derive a
statement about the entropic cone.

Before outlining the construction of the tropical cone, let us mention
that for our purposes, the language of random variables proved
inconvenient, which is why work with \term{diagrams of probability
spaces} instead.

Diagrams of probability spaces are commutative diagrams in the
category of probability spaces, with (equivalence classes of)
measure-preserving maps as morphisms, such as
\[
\tageq{diagram-examples}
\begin{cd}[column sep=small]
{}
\&
Z
\arrow{dl}
\arrow{dr}
\&
{}\\
X
\&
{}
\&
Y
\end{cd}
\quad\quad
\begin{cd}[row sep=3mm,column sep=small]
{}
\&
Z
\arrow{dl}
\arrow{dr}
\&
{}\\
X
\arrow{dr}
\&
{}
\&
Y
\arrow{dl}\\
{}
\&
U
\&
{}
\end{cd}
\quad\quad
\begin{cd}[row sep=3mm,column sep=small]
  \mbox{}
  \&
  T
  \arrow{dl}
  \arrow{d}
  \arrow{dr}
  \&
  \mbox{}
  \\
  U
  \arrow{d}
  \&
  V
  \arrow{dl}
  \arrow{dr}
  \&
  W
  \arrow{d}
  \\
  X
  \&
  Y
  \arrow[from=ul,crossing over]{}
  \arrow[from=ur,crossing over]{}
  \&
  Z
  \\
\end{cd}
\]
Collections of $n$ random variables give rise to a special type of
diagrams, that include, besides the target spaces of the random
variables themselves, the target space of every joint
variable. Such diagrams have a particular combinatorial type. The
first and the last diagrams in (\ref{eq:diagram-examples}) are
examples of such special types of diagrams in case of two and three
random variables respectively. The description of other diagrams using the language of random variables is much less transparent.

We construct the tropical cone as the \term{asymptotic cone} in the
space of diagrams of probability spaces endowed with the
\term{intrinsic entropy distance} \cite{Kovavcevic-Hardness-2012,
  Vidyasagar-Metric-2012, Matveev-Asymptotic-2018}. The asymptotic
cone captures large-scale geometry of a metric space.  As a
particularly neat application, A'Campo gave an elegant construction of
the real numbers as an asymptotic cone in a metric space of sequences
of integers \cite{ACampo-Natural-2003}.  We will
call elements in the tropical cone \term{tropical diagrams of
  probability spaces}.

The reason for the name \emph{tropical cone} is the following. For
instance in algebraic geometry, tropical varieties are, roughly
speaking, divergent sequences of classical varieties, renormalized on
a log scale with an increasing base. The adjective
`tropical' carries little semantics, but was introduced in honor of
the Brazilian mathematician and computer scientist Imre Simon
who worked on the subject of tropical mathematics. 
Analogously, we construct the asymptotic
cone from certain divergent sequences with respect to the intrinsic
entropy distance. As the intrinsic entropy distance is entropy-based,
we achieve a similar type of renormalization as in algebraic geometry.

The tropical cone has a rich algebraic structure. Indeed, we show that
it is a closed, convex cone in a Banach space.  In particular, one can
take convex combinations of tropical diagrams. Other useful operations
and constructions can be carried through for tropical diagrams,
whereas they do not have an equivalent in the classical context of
probability spaces, see~\cite{Matveev-Conditioning-2019}. All in
all, from some perspective, tropical diagrams are easier to deal with
than diagrams or probability spaces, since only rough, asymptotic
relations between probability spaces are preserved under
tropicalization, similar to how all complicated features of the
landscape disappear when looking at the Earth from outer space.

The structure of the present article is as follows. In Section
\ref{se:asymptotic-cone-abstract}, we first give a general construction
of an asymptotic cone in an abstract setting.  We believe that this
abstract setting will make the construction more transparent and
easier to follow. The results we present in that section are
probably quite standard, but we find it beneficial to
  gather them ``under one roof.'' 
In Section \ref{se:grothendieck} we show how, under certain conditions, the asymptotic cone can be interpreted as a closed convex cone in a Banach space. 
We specify to
the case of diagrams of probability spaces in Section
\ref{se:tropical-diagrams}, reformulate the Asymptotic Equipartition Property proved in \cite{Matveev-Asymptotic-2018} in terms of tropical diagrams in Section \ref{se:aep}. We conclude with a simple characterization of the tropical cone for special types of diagrams in Section \ref{se:chains}.

\section{Asymptotic Cones of Metric Abelian Monoids}
\label{se:asymptotic-cone-abstract}
In this section we define the asymptotic cone in the setting of an
abstract metric Abelian monoid. In a later section, we will specify to
the case of diagrams of probability spaces.

\subsection{Metric and pseudo-metric spaces}
A pseudo-metric space $(X,\dist)$ is a set $X$ equipped with a
pseudo-distance $\dist$, a bivariate function satisfying all the
axioms of a distance function, except that it is allowed to vanish on
pairs of non-identical points. An isometry of such spaces is a
distance-preserving map, such that for any point in the target space
there is a point in the image at zero distance away from it.  Given
such an pseudo-metric space $(X,\dist)$ one could always construct an
isometric metric space $(X/_{\dist=0}\,,\dist)$, the metric quotient,
by identifying all pairs of points that are distance zero apart.

Any property formulated in terms of the pseudo-metric holds
simultaneously for a pseudo-metric space and its metric quotient.  It
will be convenient for us to construct pseudo-metrics on spaces
instead of passing to the quotient spaces.

For a pair of points $x,y\in X$ in a pseudo-metric space $(X,\dist)$
we will write $x\aeq[\dist] y$ if $\dist(x,y)=0$. We call such a pair of
points ($\dist$-)metrically equivalent. 

Many metric-topological notions such as (Lipschitz-)continuity,
compactness, $\epsilon$-nets, dense subsets, etc., extend to the setting of a
pseudo-metric spaces and exercising certain care one may 
switch between a pseudo-metric space and its metric quotient replacing
the $\aeq[\dist]$-sign with equality.

\subsection{Metric Abelian Monoids}
A monoid is a set equipped with a bivariate associative operation and a neutral element. The operation is usually called multiplication, or addition if it is commutative.
We call a monoid with pseudo-distance $(\Gamma,+,\dist)$ a \term{metric Abelian monoid} if it satisfies: 
\begin{enumerate}
\item 
  For any $\gamma,\gamma'\in\Gamma$ holds
  \[
    \gamma + \gamma' \aeq[\dist] \gamma' + \gamma
  \]   
\item
  The binary operation is 1-Lipschitz with respect to each argument:
  For any $\gamma,\gamma',\gamma''\in\Gamma$
  \[
    \dist(\gamma+\gamma',\gamma+\gamma'')
    \leq
    \dist(\gamma',\gamma'')
  \]
\end{enumerate}

The following proposition is elementary.
\begin{proposition}{p:mam-translations}
  Let $(\Gamma,+,\dist)$ be a metric Abelian monoid. Then:
  \begin{enumerate}
  \item 
    The translations maps 
    \[
      T_{\eta}:\Gamma\to\Gamma,
      \quad
      \gamma\mapsto\gamma+\eta
    \]
    are non-expanding for any $\eta\in\Gamma$.
  \item
    For any quadruple
    $\gamma_{1},\gamma_{2},\gamma_{3},\gamma_{4}\in\Gamma$ holds
    \[
      \dist(\gamma_{1}+\gamma_{2},\gamma_{3}+\gamma_{4})
      \leq
      \dist(\gamma_{1},\gamma_{3}) + \dist(\gamma_{2},\gamma_{4})
    \]
  \item
    For every $n \in \Nbb$, and $\gamma_1, \gamma_2 \in \Gamma$ also
    holds
    \[
      \dist(n\cdot\gamma_1, n\cdot\gamma_2 ) 
      \leq 
      n\cdot\dist(\gamma_1, \gamma_2)
    \]
  \end{enumerate}
\end{proposition}

A metric Abelian monoid $(\Gamma, +, \disth)$ will be called \term{homogeneous} if it satisfies
\[\tageq{contractiondist}
  \disth(n\cdot\gamma_1, n\cdot\gamma_2) 
  =
  n\cdot\disth(\gamma_1, \gamma_2)
\]
A homogeneous metric Abelian monoid is called an
\term{$\Rbb_{\geq0}$-semi-module} $(\Gamma,+,\cdot\,,\disth)$ if in
addition there is a doubly distributive $\Rbb_{\geq0}$-action such that for any
$\lambda_{1},\lambda_{2}\in\Rbb_{\geq0}$ and
$\gamma_{1},\gamma_{2}\in\Gamma$ holds
  \begin{align*}
    \lambda_{1}\cdot(\lambda_{2}\cdot \gamma_{1})
    &\aeq[\disth]
    (\lambda_{1}\lambda_{2})\cdot \gamma_{1}
    \\
    \lambda\cdot(\gamma_{1}+\gamma_{2})
    &\aeq[\disth]
    \lambda\cdot \gamma_{1}+\lambda\cdot \gamma_{2}
    \\
    (\lambda+\lambda')\cdot \gamma_{1}
    &\aeq[\disth] 
    \lambda\cdot \gamma_{1} + \lambda'\cdot \gamma_{1}\\
    \disth(\lambda\cdot\gamma,\lambda\cdot\gamma')  
    &=
    \lambda\cdot\disth(\gamma,\gamma')
  \end{align*}

A convex cone in a normed vector space would be a typical example of
an $\Rbb_{\geq0}$-semimodule.  An intersection of a convex cone in
$\Rbb^{n}$ with the integer lattice is an example of a monoid, that
does not admit semimodule structure.

The following proposition asserts that if a metric Abelian monoid is
homogeneous, then the pseudo-distance is translation invariant, and, in
particular, it satisfies a cancellation property.  This result was
communicated to us by Tobias Fritz, see
also~\cite{Fritz-Resource}, 
\cite[Proposition 3.7]{Matveev-Asymptotic-2018}.

\begin{proposition}{p:translation-invariance}
  Let $(\Gamma,+,\disth)$ be a homogeneous metric Abelian monoid.
  Then the pseudo-distance function $\disth$ is translation invariant, that is
  it satisfies for any $\gamma_{1},\gamma_{2},\eta\in\Gamma$ 
  \[
  \disth(\gamma_{1}+\eta,\gamma_{2}+\eta)
  =
  \disth(\gamma_{1},\gamma_{2})
  \]
  In particular, the following cancellation property holds in $\Gamma$
  \begin{quote}
    If $\gamma_{1}+\eta \aeq[\disth] \gamma_{2}+\eta$, then
    $\gamma_{1}\aeq[\disth]\gamma_{2}$.
  \end{quote}
\end{proposition}

\subsection{Asymptotic Cones (Tropicalization) of Monoids}
In our construction points of the asymptotic cone of
$(\Gamma,+,\dist)$ will be sequences of points in $\Gamma$ that grow
almost linearly in a certain sense described below.

\subsubsection{Admissible functions} Admissible functions will be used
to measure the deviation of a sequence from being linear. 
We call a
function $\phi:\Rbb_{\geq1}\to\Rbb_{\geq0}$ \term{admissible} if 
\begin{enumerate}
\item 
  the function $\phi$ is non-decreasing;
\item
  \label{i:admissible-integral-bound} 
  there exists a constant $D_{\phi}\geq0$ such that
  $s\cdot\int_{s}^\infty \frac{\phi(t)}{t^2} \d t \leq
  \frac{D_{\phi}}{8}\cdot\phi(s)$ for any $s\geq1$. In particular the
  function $\phi$ is summable against $\d t/t^{2}$.
\end{enumerate}
For example, the function
$\phi(t)\!\!:=\!\!t^{\alpha}$ is admissible for any $0\leq\alpha<1$. Any
admissible function is necessarily sub-linear, that is $\phi(t)/t\to0$
as $t\to\infty$. A linear combination of admissible functions with
non-negative coefficients is also admissible.

\subsubsection{Quasi-linear sequences}
Let $(\Gamma,+,\dist)$ be a metric Abelian monoid and $\phi$ be an
admissible function. A sequence
$\bar\gamma=\set{\gamma(i)}\in\Gamma^{\Nbb_{0}}$ will be called
quasi-linear with defect bounded by $\phi$ if
for every $m, n \in \Nbb$ the following bound is satisfied
\[
\dist\big( \gamma(m+n), \gamma(m) + \gamma(n) \big)
\leq \phi(m + n)
\]
For technical reasons we also require $\gamma(0)=0$.  Sequences that
are quasi-linear with defect bounded by $\phi\equiv0$
will be called \term{linear sequences}.

For an admissible function $\phi$ we will write $\qlin_\phi(\Gamma,
\dist)$ for the space of all quasi-linear sequences with defect
bounded by $C\cdot\phi$ for some (depending on the sequence) constant
$C\geq 0$.
We will also use notation $\lin(\Gamma,
\dist):=\qlin_{0}(\Gamma,\dist)$ for the space of linear sequences.

\subsubsection{Asymptotic distance}
Given two quasi-linear sequences $\bar{\gamma}_1\in
\qlin_{\phi_1}(\Gamma, \dist)$ and $\bar{\gamma}_2 \in
\qlin_{\phi_2}(\Gamma,\dist)$ the sequence of distances $a(n)
:= \dist(\gamma_1(n), \gamma_2(n))$ is $\phi_{3}$-subadditive, where
$\phi_{3}=\phi_{2}+\phi_{2}$ is also admissible, i.e.
\[
a(m+n)\leq a(n) + a(n) + \phi_{3}(n+m)
\]
for any $n,m\in\Nbb$.  By the generalization of Fekete's Lemma by De
Bruijn and Erd\"os \cite[Theorem 23]{Bruijn-Some-1952}, it follows
that the following limit exists and finite
\[
\dista(\bar{\gamma_1},\bar{\gamma_2}) 
:= \lim_{n \to \infty} \frac{1}{n} \dist( \gamma_1(n), \gamma_2(n) )
\]

We call the quantity $\dista(\bar\gamma_{1},\bar\gamma_2)$ the
asymptotic distance between
$\bar\gamma_{1},\bar\gamma_{2}\in\qlin_{\phi}(\Gamma,\dist)$. It is
easy to verify that $\dista$ indeed satisfies all axioms of a
pseudo-distance.  Even if $\dist$ was a proper distance function, the
corresponding asymptotic distance may vanish on some pairs of
non-identical elements.  We call two sequences
$\bar\gamma_{1}\in\qlin_{\phi_{1}}(\Gamma,\dist)$,
$\bar\gamma_{2}\in\qlin_{\phi_{2}}(\Gamma,\dist)$ \term{asymptotically
  equivalent} if $\dista(\bar\gamma_{1},\bar\gamma_{2})=0$ and write
\[
  \bar \gamma_{1} \aeq[\dista]  \bar \gamma_2
\]

\subsubsection{Quasi-homogeneity}
We will show that quasi-linear sequences are also quasi-homogeneous in
the sense  of the following lemma.
\begin{lemma}{l:scaling-of-iteration}
  Let $\bar{\gamma}\in\Gamma^{\Nbb_{0}}$ be a sequence with
  $\phi$-bounded defect.  Then for any $m, n \in \Nbb$
  \[
  \dist( \gamma(m \cdot n) , m \cdot \gamma(n) ) 
  \leq 
  8 \cdot m \cdot n \cdot 
  \int_{n}^{2m\cdot n}\frac{\phi(t)}{t^{2}}\d t
  \]
\end{lemma}	
\begin{proof}
  Define the function $\psi$ related to $\phi$ as follows
  \[
  \psi(s):=\phi(\ebf^{s})/\ebf^{s}
  \quad\text{or}\quad
  \phi(t)=:t\cdot\psi(\ln t)
  \]
  The conclusion of the lemma in terms of $\psi$ then reads
  \[
  \dist( \gamma(m \cdot n) , m \cdot \gamma(n) ) 
  \leq 
  8 \cdot m \cdot n \cdot 
  \int_{\ln n}^{\ln(2\cdot m\cdot n)}\psi(s)\d s
  \]
  and it is in that form it will be proven below.  
  
  Due to monotonicity properties of $\phi$ function $\psi$ satisfies,
  for any $0\leq s_{0}\leq s$ 
  \begin{align*}
    \psi(s_{0})
    &\leq
    \psi(s)\cdot\ebf^{s-s_{0}}
    \\
    \tageq{psi-integral}
    \psi(s_{0})
    &\leq
    4\int_{s_{0}}^{s_{0}+\ln2}\psi(s)\d s
  \end{align*}

  We proceed by induction with respect to $m$, keeping $n$ fixed.
  The conclusion of the lemma is obvious for $m=1$. For the induction step
  let $m=2m'+\epsilon\geq 2$, where $m'=\lfloor m/2\rfloor$ and
  $\epsilon\in\set{0,1}$. 
  Then using bound~(\ref{eq:psi-integral}) we estimate
  \begin{align*}
    \dist\big(&\, \gamma(m \cdot n)\,,\, m \cdot \gamma(n) \,\big)
    \\
    &=
    \dist\big(\, 
       \gamma(m'\cdot n + m'\cdot n + \epsilon \cdot n) 
       \,,\, 
        m'\cdot \gamma(n)+m'\cdot \gamma(n)+\epsilon\cdot \gamma(n) \,\big)
    \\
    &\leq
    2\dist\big(\, \gamma(m' \cdot n) \,,\, m' \cdot \gamma(n) \,\big)
    +
    2\phi\big(\,m\cdot n\,\big)
    \\
    &\leq
    16 m'\cdot n\cdot\int_{\ln n}^{\ln(2m'\cdot n)}\psi(s)\d s 
    + 2m\cdot n\cdot\psi\big(\ln(m\cdot n)\big)
    \\
    &\leq
     8m\cdot n
     \left(
        \int_{\ln n}^{\ln(2m'\cdot n)}\psi(s)\d s
        +
        \int_{\ln(m\cdot n)}^{\ln(2m\cdot n)}\psi(s)\d s
     \right)
     \leq
     8 m\cdot n\cdot\!\!\!\int_{\ln n}^{\ln(2m\cdot n)}\psi(s)\d s
  \end{align*}
\end{proof}
Applying bound~(\ref{i:admissible-integral-bound}) in the definition
of admissible functions on page~\pageref{i:admissible-integral-bound}
we obtain the following corollary.
\begin{corollary}{p:psi-homo}
  Let $\bar{\gamma}$ be a sequence with $\phi$-bounded defect. 
  Then for any $m, n \in \mathbb{N}$
  \[
  \dist( \gamma(m \cdot n) , m \cdot \gamma(n) ) 
  \leq 
  8 \cdot m \cdot n \cdot 
  \int_{n}^{\infty}\frac{\phi(t)}{t^{2}}\d t
  \leq D_\phi \cdot m\cdot\phi(n)
  \]
\end{corollary}
\subsubsection{The semi-module structure}
The group operation $+$ on $\Gamma$ induces a $\dista$-continuous (in
fact, 1-Lipschitz) group operation on $\qlin_{\phi}(\Gamma,\dist)$ by
adding sequences element-wise. Thus
$(\qlin_{\phi}(\Gamma,\dist),+,\dista)$ is also a metric Abelian
monoid. In addition, it carries the structure of a
$\Rbb_{\geq0}$-semi-module, as explained below.

The validity of the following constructions is very easy to verify, so
we omit the proofs.  Let $\phi>0$ be an admissible function.  The set
$\qlin_{\phi}(\Gamma,\dist)$ admits an action of the multiplicative
semigroup $(\Rbb_{\geq0},\,\cdot\,)$ defined in the following way. Let
$\lambda\in\Rbb_{\geq 0}$ and
$\bar\gamma=\set{\gamma(n)}\in\qlin_{\phi}(\Gamma,\dist)$.  Then
define the action of $\lambda$ on $\bar\gamma$ by
\[\tageq{R-action} 
  {\lambda\cdot\bar\gamma}
  := 
  \set{\gamma\big(\lfloor\lambda\cdot n\rfloor\big)}_{n\in\Nbb_{0}} 
\]
This is only an action up to asymptotic equivalence.  Similarly, in
the constructions that follow we are tacitly assuming they are valid
up to asymptotic equivalence.

The action 
\[
  \cdot:\Rbb_{\geq 0}\times\qlin_{\phi}(\Gamma,\dist)
  \to
  \qlin_{\phi}(\Gamma,\dist)
\]
is continuous with respect to $\dista$ and, moreover it is
a homothety (dilation), that is
\[
  \dista(\lambda\cdot\bar\gamma_{1},
         \lambda\cdot\bar\gamma_{2})
  =
  \lambda\cdot\dista(\bar\gamma_{1},\bar\gamma_{2})
\]
The semigroup structure on $\qlin_{\phi}(\Gamma,\dist)$
is distributive with respect to the $\Rbb_{\geq 0}$-action
\begin{align*}
  \lambda\cdot(\bar\gamma_{1}+\bar\gamma_{2})
  &=
  \lambda\cdot\bar\gamma_{1}+
  \lambda\cdot\bar\gamma_{2}
  \\
  (\lambda_{1}+\lambda_{2})\cdot\bar\gamma
  &\aeq[\dista]
  \lambda_{1}\cdot\bar\gamma+
  \lambda_{2}\cdot\bar\gamma
\end{align*}

In particular, for $n \in \mathbb{N}$ and $\bar\gamma\in\qlin_{\phi}(\Gamma,\dist)$
\[
  \underbrace{\bar{\gamma}+\cdots+\bar{\gamma}}_{n} 
  \aeq[\dista] 
  n\cdot\bar\gamma
\]
  
\subsubsection{Completeness}
Here, we introduce additional conditions on a metric Abelian monoid
$(\Gamma,+,\dist)$, that guarantee that $(\qlin_{\phi}(\Gamma),\dista)$ is a
complete metric space.

Suppose $\phi$ is an admissible function and $(\Gamma,+,\dist)$ is a
metric Abelian monoid satisfying the following additional property: there exists a constant $C>0$, such that for
any quasi-linear sequence $\bar\gamma\in\qlin_\phi(\Gamma,\dist)$,
there exists an asymptotically equivalent quasi-linear sequence
$\bar\gamma'$ with defect bounded by $C \phi$. 
Note that, contrary to the situation in the definition of $\qlin_\phi(\Gamma, \dist)$, the constant $C$ is now not allowed to depend on the sequence.
If this is the case, we
say that $\qlin_{\phi}(\Gamma,\dist)$ has the
($C$-)\term{uniformly bounded defect property}.

\begin{proposition}{p:boundeddefect}
  Suppose a metric Abelian monoid $(\Gamma,+,\disth)$ and an
  admissible function $\phi>0$ are such that
  $(\qlin_\phi(\Gamma,\disth),\distha)$ has the uniformly bounded
  defect property and the distance function $\disth$ is
  homogeneous. Then the space $(\qlin_\phi(\Gamma,\disth),\distha)$ is
  complete.
\end{proposition}
\begin{proof}
  Given a Cauchy sequence $\set{\bar\gamma_{i}}$ of elements in
  $(\qlin_\phi(\Gamma,\disth),\distha)$ we need to find a limit element
  $\bar\eta\in\qlin_\phi(\Gamma,\disth)$.  
  We will construct $\bar\eta$ by a diagonal argument.
  First we replace each element
  of the sequence $\set{\bar\gamma_{i}}$ by an asymptotically
  equivalent element with defect bounded by $C \phi$
  according to the assumption of the proposition. We
  will still call the new sequence $\set{\bar\gamma_{i}}$. In fact, we
  may without loss of generality assume that $C=1$.

  We begin by establishing a bound on the divergence of the tails of
  sequences $\bar\gamma_{i}$ and $\bar\gamma_{j}$.  By homogeneity of
  $\disth$ and Corollary \ref{p:psi-homo}, it holds for any
  $n,k\in\Nbb$ that
  \begin{align*}
    k \cdot \disth\big(\gamma_{i}(n),\gamma_{j}(n)\big)
    &=
    \disth\big(k\cdot\gamma_{i}(n),k\cdot\gamma_{j}(n)\big)
    \\
    &\leq
    \disth\big(\gamma_{i}(k\cdot n),\gamma_{j}(k\cdot n)\big) 
    + 
    2k \cdot D_{\phi}\cdot \phi(n)    
  \end{align*}
  Dividing by $k$ and passing to the limit $k\to\infty$, while keeping
  $n$ fixed, we obtain
  \[
    \disth(\gamma_{i}(n),\gamma_{j}(n))
    \leq
    n\cdot\distha(\bar\gamma_{i},\bar\gamma_{j}) 
    + 
    2D_{\phi}\cdot\phi(n)
  \]	 
  Since the sequence $(\bar\gamma_{i})_{i\in\Nbb_{0}}$ is Cauchy, it
  follows that for any $n \in \Nbb$ there is a number
  $\ibf(n)\in\Nbb$ such that for any $i,j\geq\ibf(n)$ holds
  \[
  \distha(\bar\gamma_{i},\bar\gamma_{j})\leq \frac{1}{n}
  \]
  Then for any $i,j,n\in\Nbb$ with $i,j\geq \ibf(n)$ we have
  the following bound
  \[\tageq{boundedmembers}
  \disth\big(\gamma_{i}(n),\gamma_{j}(n)\big)
  \leq
  2D_{\phi} \cdot \phi(n) + 1
  \]	
  Now we are ready to define the limiting sequence $\bar\eta$ by
  setting
  \[
  \eta(n):=\gamma_{\ibf(n)}(n)
  \]
  First we verify that $\bar\eta$ is quasi-linear. For $m, n \in
  \Nbb$, we have 
  \begin{align*}
    \disth
    \big(
      \eta(n+m),
      &\eta(n)+\eta(m)
    \big)
    =
    \disth
    \big(\,
      \gamma_{\ibf(n+m)}(n+m) \,,\,
      \gamma_{\ibf(n)}(n)+\gamma_{\ibf(m)}(m)
    \,\big)
    \\
    &\leq
    \disth
    \big(\,
      \gamma_{\ibf(n+m)}(n+m) \,,\,
      \gamma_{\ibf(n+m)}(n) + \gamma_{\ibf(n+m)}(m)\,
    \big)\;+
    \\
    &\quad
    \disth
    \big(\,
      \gamma_{\ibf(n+m)}(n)+\gamma_{\ibf(n+m)}(m) \,,\,
      \gamma_{\ibf(n)}(n)+\gamma_{\ibf(m)}(m)\,
    \big)
    \\
    &\leq
    \phi(n+m) 
    + 
    2D_{\phi}\cdot\phi(n) + 1 +
    2D_{\phi}\cdot\phi(m) + 1
    \\
    &\leq 
    (4D_{\phi}+1)\phi(n+m) + 2
    \leq C' \cdot \phi(n + m)
  \end{align*}
  for some constant $C'> 0$.

  The convergence of $\bar\gamma_{i}$ to $\bar\eta$ is shown as
  follows. For $n,k\in\Nbb$ let $q_n,r_n\in\Nbb_{0}$ be the quotient and
  the remainder of the division of $n$ by $k$, that is $n=q_n\cdot k+r_n$
  and $0\leq r_n < k$.  Fix $k\in\Nbb$ and let $i\geq\ibf(k)$, then
  \begin{align*}
    \distha(\bar\gamma_{i},\bar\eta)
    &=
    \lim_{n\to\infty}\frac1n
    \disth\big(\gamma_{i}(n),\eta(n)\big)
    \\
    &=
    \lim_{n\to\infty}
    \frac1n
    \disth
    \big(\,
      \gamma_{i}(q_n\cdot k+r_n) \,,\,
      \gamma_{\ibf(n)}(q_n\cdot k+r_n)\,
    \big)
    \\
    &\leq
    \limsup_{n\to\infty}
    \frac1n
    \Big(\rule{0mm}{5mm}
      q_n\cdot
      \disth\big(\gamma_{i}(k),\gamma_{\ibf(n)}(k)\big)+
      \disth\big(\gamma_{i}(r_n),\gamma_{\ibf(n)}(r_n)\big) \;+
    \\
    &\quad\quad\quad\quad\quad\;\;+
      4q_n \cdot D_{\phi} \cdot \phi(k)+2 \phi(n)
    \Big)
    \\
    &\leq
    \limsup_{n\to\infty}
    \frac1n
    \Big(\rule{0mm}{5mm}
    q_n \cdot (2D_\phi\cdot\phi(k) + 1) 
    + 
    (2D_\phi\cdot\phi(r_n) + 1) 
    \;+
    \\
    &\quad\quad\quad\quad\quad\;\;+
    4q_n \cdot D_\phi \cdot \phi(k)+2 \phi(n)
    \Big)
    \\
    &=C'' \cdot  \phi(k)/k
  \end{align*}
  Since $k\in\Nbb$ is arbitrary and $\phi$ is sub-linear we have 
  \[
  \lim_{i\to\infty}\distha(\bar\gamma_{i},\bar\eta)=0
  \]
\end{proof}

\subsubsection{On the density of linear sequences}\hspace{-1pt}
For a metric Abelian monoid $(\Gamma, +, \dist)$ together with an
admissible function  $\phi$ we say that $\qlin_{\phi}(\Gamma,\dist)$ has the
\term{vanishing defect property} if for every $\epsilon > 0$ and for
every $\bar\gamma \in \qlin_\phi(\Gamma, \dist)$ there exists an
asymptotically equivalent quasi-linear sequence $\bar\gamma'$ with
defect bounded by another admissible function $\psi$ such that
$\int_{1}^{\infty}\frac{\psi(t)}{t^{2}}\d t<\epsilon$.

The proposition below gives a sufficient condition under which the
linear sequences are dense in the space of quasi-linear sequences.

\begin{proposition}{p:eps-linear-dense}
  Suppose $(\Gamma, + , \dist)$ and admissible function $\phi$ have the vanishing defect property. Then
  $\lin(\Gamma,\dist)$ is dense in $(\qlin_\phi(\Gamma,\dist),\dista)$.
\end{proposition}
\begin{proof}
  Let $\bar\gamma=\set{\gamma(n)}$ be a quasi-linear sequence.  For
  any $i\in\Nbb$ select a sequence $\bar\gamma_{i}$ asymptotically
  equivalent to $\bar\gamma$ with defect bounded by an admissible
  function $\phi_{i}$ such that
  $\int_{1}^{\infty}\frac{\phi_{i}(t)}{t^{2}}\d t<1/i$ according to
  the ``vanishing defect'' assumption of the lemma.
  
  Define $\bar{\eta}_i$ by
  \[
    \eta_i(n) := n \cdot \gamma_i(1)
  \]
  Then
  \begin{align*}
    \dista(\bar\gamma,\bar\eta_{i})
    &=
    \dista(\bar\gamma_{i},\bar\eta_{i})
    =
    \lim_{n\to\infty}
    \frac1n
    \dist(\gamma_{i}(n),\eta_{i}(n))
    =
    \lim_{n\to\infty}
    \frac1n
    \dist\big(\gamma_{i}(n),n\cdot\gamma_{i}(1)\big)
    \\
    &\leq
    8\int_{1}^{\infty}\frac{\phi_{i}(t)}{t^{2}}\d t
    \leq \frac{8}{i}
  \end{align*}  	
Thus, any quasi-linear sequence can be approximated by linear
sequences. 
\end{proof}

\subsubsection{Asymptotic distance on original monoid}
\label{suse:metric-original-group}
Starting with an element $\gamma\in\Gamma$ one can construct a linear
sequence $\vec\gamma=\set{i\cdot\gamma}_{i\in\Nbb_{0}}$. In view of
Proposition~\ref{p:mam-translations}, the map
\[\tageq{inclusions-abstract}
  \vec\cdot:\big(\Gamma, \dist \big) \to \big(\lin(\Gamma,\dist), \dista\big)
\]
is a contraction.

 By the inclusions in (\ref{eq:inclusions-abstract}) we have an induced metric
$\disth$ on $\Gamma$, satisfying for any
$\gamma_{1},\gamma_{2}\in\Gamma$
\[\tageq{deltalessd}
  \disth(\gamma_{1},\gamma_{2})\leq\dist(\gamma_{1},\gamma_{2})
\]
and the following homogeneity condition
\[\tageq{deltalin}
  \disth(n\cdot\gamma_{1},n\cdot\gamma_{2})
  =
  n\cdot\disth(\gamma_{1},\gamma_{2})
\]
for all $n \in \Nbb_0$.

Note that if $\dist$ was homogeneous to begin with, then
$\disth$ coincides with $\dist$ on $\Gamma$.

By virtue of the bound $\disth \leq \dist$, sequences that are
quasi-linear with respect to $\disth$ are also quasi-linear with
respect to $\dist$.  Since $\disth$ is scale-invariant, the associated
asymptotic distance $\distha$ coincides with $\disth$ on $\Gamma$. We
will show (in Lemma \ref{p:dist-dista-isometry-dense} below) that
$\distha$ also coincides with $\dista$ on $\dist$-quasi-linear
sequences.

Let $\phi$ be an admissible function. In order to organize all these
statements, and to be more precise, let us include the spaces in the
following commutative diagram.
\[\tageq{tropical-diagram-0}
  \begin{cd}[row sep=tiny]
    \mbox{}
    \&
    \big(\lin(\Gamma,\dist), \dista \big)
    \arrow[hookrightarrow]{dd}{\i_{1}}
    \arrow[hookrightarrow]{r}{\j_{1}}
    \&
    \big(\qlin_{\phi}(\Gamma,\dist), \dista \big)
    \arrow[hookrightarrow]{dd}{\i_{2}}
    \\
    (\Gamma,\dist)
    \arrow{ru}{f}
    \arrow{rd}{f'}
    \\
    \mbox{}
    \&
    \big(\lin(\Gamma,\disth),\distha\big)
    \arrow[hookrightarrow]{r}{\j_{2}}
    \&
    \big(\qlin_{\phi}(\Gamma,\disth),\distha\big)
  \end{cd}
  \]

  The maps $f, f'$ and $\i_1$ are isometries.  The maps $\j_1$ and
  $\j_2$ are isometric embeddings.  The next lemmas show that $\i_2$
  is also an isometric embedding, and it has dense image.

  \begin{lemma}{p:dist-dista-isometry-dense}
    Let $\phi$ be a positive, admissible function. Then, the natural inclusion 
    \[
    \i_{2}:\big(\qlin_{\phi}(\Gamma,\dist),\dista\big)
    \into
    \big(\qlin_{\phi}(\Gamma,\disth),\distha\big)
    \]
    is an isometric embedding with the dense image.
  \end{lemma}

\begin{proof}
  First we show that the map $\i_{2}$ is an isometric embedding.  Let
  $\bar\gamma_{1},\bar\gamma_{2}\in\qlin_{\phi}(\Gamma,\dist)$ be two
  $\phi$-quasi-linear sequences with respect to the distance function
  $\dist$. We have to
  show that the two numbers
  \[
    \dista(\bar\gamma_{1},\bar\gamma_{2})
    =
    \lim_{n\to\infty}\frac1n \dist\big(\gamma_{1}(n),\gamma_{2}(n)\big)
  \]
  and
  \[
    \distha(\bar\gamma_{1},\bar\gamma_{2})
    =
    \lim_{n\to\infty}\frac1n \disth\big(\gamma_{1}(n),\gamma_{2}(n)\big)
  \]
  are equal.  Since shifts are non-expanding maps, we have
  $\disth\leq\dist$ and it follows immediately that
  \[
    \distha(\bar\gamma_{1},\bar\gamma_{2})
    \leq
    \dista(\bar\gamma_{1},\bar\gamma_{2})
  \]
  and we are left to show the opposite inequality.
  We will do it as follows. Fix $n>0$, then
  \begin{align*}
    \dista(\bar\gamma_{1},\bar\gamma_{2})
    &=
    \lim_{k\to\infty}
    \frac{1}{k\cdot n}
    \dist\big(\gamma_1(k\cdot n),\gamma_{2}(k\cdot n)\big)
    \\
    &\leq
    \lim_{k\to\infty}\frac{1}{k\cdot n}
    \bigg(
      \dist\big(k\cdot\gamma_1(n),k\cdot\gamma_{2}(n)\big)
      +
      2 k\cdot D_\phi \cdot \phi(n)
    \bigg)
    \\
    &\leq
    \frac1n\dista\big(\gamma_{1}(n),\gamma_{2}(n)\big)
    +
    2D_{\phi}\frac{\phi(n)}{n}
  \end{align*}	
  Passing to the limit with respect to $n$ gives the required inequality
  \[
  \dista(\bar\gamma_{1},\bar\gamma_{2})
  \leq
  \distha(\bar\gamma_{1},\bar\gamma_{2})
  \]
  
  Now we will show that the image of $\i_{2}$ is dense.
  Given an element $\bar\gamma=\set{\gamma(n)}$ in
  $\qlin_{\phi}(\Gamma,\dista)$ we have to find a $\distha$-approxi\-ma\-ting
  sequence $\bar\gamma_{i}=\set{\gamma_{i}(n)}$ in
  $\qlin_{\phi}(\Gamma,\dist)$.
  Define 
  \[
    \gamma_{i}(n)
    :=
    \lfloor\frac{n}{i}\rfloor\cdot\gamma(i)
  \]
  We have to show that each $\bar\gamma_{i}$ is $\dist$-quasi-linear
  and that $\distha(\bar\gamma_{i},\bar\gamma)\too[i\to\infty]0$.
  These statements follow from
  \begin{align*}
    \dist\big(\gamma_i(m+n), \gamma_i(m) + \gamma_i(n) \big)
    &=
    \dist
    \left( 
      \left\lfloor \frac{m+n}{i}\right\rfloor\cdot\gamma(i), 
      \left\lfloor \frac{m}{i}\right\rfloor\cdot\gamma(i)
      +
      \left\lfloor \frac{n}{i} \right\rfloor\cdot\gamma(i)
    \right)
    \\
    &\leq 
    \dist \left( \gamma(i), \mathbf 0 \right)\\
    &\leq C_i \cdot \phi(m+n)
  \end{align*} 
  for some $C_i > 0$. It is worth noting that the defect of $\bar\gamma_{i}$ may not be bounded uniformly with respect to $i$. Finally, it holds that
  \begin{align*}
    \distha(\bar\gamma_i,\bar\gamma)
    &= 
    \lim_{n \to \infty} \frac1n\disth
    \left( 
    \gamma_i(n), \gamma(n)
    \right) 
    = 
    \lim_{n \to \infty} \frac1n \disth 
    \left( 
    \left\lfloor\frac{n}{i}\right\rfloor\cdot\gamma(i), 
    \gamma(n) 
    \right)
    \\
    &\leq 
    \lim_{n \to \infty} 
    \left[
      \frac1n \disth 
      \left( 
      \gamma\left(i\lfloor\tfrac{n}{i}\rfloor\right), 
      \gamma(n) 
      \right) 
      +
      \frac{1}{n}\left\lfloor \frac{n}{i} \right\rfloor \cdot D_\phi \cdot \phi(i)
    \right]
    \\
    &\leq 
    \lim_{n \to \infty}
    \left[ 
      \frac1n 
      \max_{k=0,\ldots,i-1} \disth\left( \gamma(k), {\bm0} \right)
      + 
      \frac{1}{n}\phi(n) 
    \right]
    + 
    D_{\phi}\frac{\phi(i)}{i}
    = 
    D_{\phi}\frac{\phi(i)}{i}\stackrel{i\to\infty}{\too}0
  \end{align*}
\end{proof}
The difference between two distance functions $\dista$ and $\distha$
is very small: $\dista$ is defined on the dense subset of the domain
of definition of $\distha$ and they coincide whenever are both
defined. From now on we will not use the notation $\distha$.

\section{Grothendieck construction}
\label{se:grothendieck}
Given an Abelian monoid with a cancellation property, there is a
minimal Abelian group (called the Grothendieck Group of the
monoid), into which it isomorphically embeds. 
Similarly, an $\Rbb_{\geq0}$-semi-module naturally embeds into a
normed vector space. A nice example of this construction applied to
the semi-module of convex sets in $\Rbb^n$ (with the Minkowski sum and
the Hausdorff distance) can be found
in \cite{Raadstrom-Embedding-1952}.

\begin{proposition}{p:banach-embedding}
  Let $(\Gamma,+,\cdot,\disth)$ be a complete metric Abelian monoid
  with $\Rbb_{\geq0}$ action (an $\Rbb_{\geq0}$-semi-module) with
  homogeneous pseudo-metric $\disth$. Then there exists a Banach space
  $(\Bbf,||\,\cdot\,||)$ and a distance-preserving homomorphism \[
  f: \Gamma\to \Bbf \] such that the image of $f$ is a closed convex
  cone.
\end{proposition}
If $\dist$ is a proper pseudo-metric (not a metric), then
the map $f$ is not injective.

\begin{proof} By Lemma~\ref{p:translation-invariance} the pseudo-metric
$\deltabf$ is translation invariant. 
We can therefore apply the
Grothendieck construction to define a normed vector space $\Bbf_0$: Define
\[
  \Bbf_0
  :=
  \set{(x,y)\st x,y\in\Gamma}/\sim
\]
where $(x,y)\sim(x',y')$ if there are $z,z'\in\Gamma$, such that
$(x+z,y+z)\aeq[\dist](x'+z',y'+z')$.

Define also addition, multiplication by a scalar and a norm on $\Bbf_0$ by setting for
all $x,y,x',y'\in\Gamma$ and $\lambda\in\Rbb$
\begin{align*}
  (x,y)+(x',y')
  &:=
  (x+x',y+y')
  \\
  (-1)\cdot(x,y)
  &:=
  (y,x)
  \\
  \lambda\cdot(x,y)
  &:=
  \sign(\lambda)\cdot(|\lambda|\cdot x,|\lambda|\cdot y)
  \\
  ||(x,y)||
  &:=
  \deltabf(x,y)
\end{align*}

These operations respect the equivalence
relation and turn $(\Bbf_0,+,\cdot,||\,\cdot\,||)$ into a normed
vector-space.  The map $f$ defined by
\[
  f:\Gamma\to\Bbf_0,
  \quad
  x\mapsto (x,\mathbf0)
\]
is a well-defined distance-preserving homomorphism.

That $f(\Gamma)$ is closed immediately follows as $\Gamma$ is complete
and $f$ is distance-preserving.

In general, the space $\Bbf_0$ is not complete. We define the Banach
space $\Bbf$ as the completion of the normed vector space $\Bbf_0$.
\end{proof}

\section{Tropical probability spaces and their diagrams}
\label{se:tropical-diagrams}
\subsection{Diagrams of probability spaces.}
We will now briefly describe the construction of diagrams of
probability spaces, see \cite{Matveev-Asymptotic-2018} for a more
detailed discussion.  By a \term{finite probability space} we will mean a set
(not necessarily finite) with a probability measure, such that the
support of the measure is finite. 
For such probability space $X$ we denote by $|X|$ the cardinality of the
support of probability measure and the expression $x\in X$ will mean,
that $x$ is an \term{atom} in $X$, which is a point of positive weight
in the underlying set. 

We will consider commutative
diagrams of finite probability spaces, where arrows are equivalence
classes of measure-preserving maps. Two maps are considered equivalent
if they coincide on a set of full measure and such equivalence
  classes will be called \term{reductions}.

Three examples of diagrams of probability spaces are pictured
  in (\ref{eq:diagram-examples}).  The combinatorial structure of such
  a commutative diagram can be recorded by an object $\Gbf$, which
  could be equivalently considered as a special type of category, a
  finite poset, or a directed acyclic graph (DAG) with additional
  properties. We will call such objects simply \term{indexing
    categories}. Below we briefly recall the definition.

An \term{indexing category} is a finite category such that for any pair of
objects there exists at most one morphism between them in either
direction, and such that it satisfies the following property. For any pair of
objects $i,j$ in an indexing category $\Gbf$ there exists a \term{least
common ancestor}, i.e.~an object $k$ such that there are morphisms
$k\to i$ and $k\to j$ in $\Gbf$ and such that for any other object $l$
admitting morphisms $l\to i$ and $l\to j$, there is also a morphism
$l\to k$.

By $\size{\Gbf}$ we denote the number of objects in the indexing category, or
equivalently the number of vertices in the DAG or the number of points
in the poset $\Gbf$.  Important class of examples of  indexing categories
are so called \term{full categories} $\Lambdabf_{n}$, that correspond
to the poset of non-empty subsets of a set $\set{1,\ldots,n}$ ordered
by inclusion.  If $n=2$, we call the category
\[
\Lambdabf_{2}=(O_{1}\ot O_{\set{1,2}}\to O_{2})
\]
a fan.

The space of all commutative diagrams of a fixed combinatorial type
will be denoted $\prob\langle\Gbf\rangle$.  A morphism between two
diagrams $\Xcal,\Ycal\in\prob\langle\Gbf\rangle$ is defined to be the
collection of morphisms between corresponding individual spaces in
$\Xcal$ and $\Ycal$, that commute with morphisms within the diagrams
$\Xcal$ and $\Ycal$. 

The construction of forming commutative diagrams could be
iterated, producing diagrams of diagrams. Especially important will be
two-fans of $\Gbf$-diagrams, the space of which will be denoted
$\prob\<\Gbf\>\<\Lambdabf_{2}\>$. 

A two-fan $\Xcal$ will be called
\term{minimal}, if for any morphism of $\Xcal$ to another two-fan
$\Ycal$, the following holds: if the induced morphisms on the feet are
isomorphisms, then the top morphism is also an isomorphism. Any $\Gbf$-diagram will be called minimal if for any sub-diagram, which is a
two-fan, it contains a minimal two-fan with the same feet.

Given an $n$-tuple $(\Xsf_{1},\ldots,\Xsf_{n})$ of finite-valued random
variables, one can construct a minimal $\Lambdabf_{n}$-diagram
$\Xcal=\set{X_{I};\chi_{IJ}}$ by setting for
any $\emptyset\neq I\subset\set{1,\ldots,n}$
\[
  X_{I}=\prod_{i\in I} X_{i}
\]
where $X_{i}$ is the target space of random variable $\Xsf_{i}$, and
the probabilities are the induced distributions. For the diagram
constructed in such a way we will write
$\Xcal=\<\Xsf_{1},\ldots,\Xsf_{n}\>$.  On the other hand, any
$\Lambdabf_{n}$-diagram gives rise to the $n$-tuple of random
variables with the domain of definition being the initial space and
the targets being the terminal spaces. 

The tensor product $\Xcal\otimes\Ycal$ of two $\Gbf$-diagrams is defined
by taking the tensor product of corresponding probability spaces and
the Cartesian product of maps.

The special $\Gbf$-diagram in which all the spaces are
isomorphic to a single probability space $X$ will be denoted by
$X^{\Gbf}$.

For a diagram
$\Xcal\in\Prob\langle\Gbf\rangle$ one can evaluate entropies of the
individual spaces. The corresponding map will be denoted
\[
\ent_{*}:\prob\langle\Gbf\rangle\to\Rbb^{\Gbf}
\]
where the target space is the space of $\Rbb$-valued functions on
objects in $\Gbf$ and it is equipped with the $\ell^{1}$-norm.

For a two-fan $\Fcal=(\Xcal\ot\Zcal\to\Ycal)$ of $\Gbf$-diagrams
define the \term{entropy distance}
\[
  \kd(\Fcal)
  :=
  \|\ent_{*}\Zcal-\ent_{*}\Xcal\|_{1}
  +
  \|\ent_{*}\Zcal-\ent_{*}\Xcal\|_{1}
\]
We interpret $\kd(\Fcal)$ as a measure of deviation of $\Fcal$ from
being an isomorphism between the diagrams $\Xcal$ and $\Ycal$. Indeed,
$\kd(\Fcal)=0$ if and only if the two morphisms in $\Fcal$ are
isomorphisms.

We define the \term{intrinsic entropy distance} $\ikd$ on the space $\prob\langle\Gbf\rangle$
by
\[
  \ikd(\Xcal,\Ycal)
  :=
  \inf\set{\kd(\Fcal)\st \Fcal=
    (\Xcal\ot\Zcal\to\Ycal)\in
    \prob\langle \Gbf \rangle\langle \Lambdabf_{2} \rangle
  }
\]

The tensor product is 1-Lipschitz with respect to $\ikd$, thus
$(\prob\langle \Gbf \rangle,\otimes,\ikd)$ is a metric Abelian monoid
and $\ent_{*}:(\prob\langle \Gbf
\rangle,\otimes,\ikd)\to(\Rbb^{\Gbf},\|\cdot\|_{1})$ is a 1-Lipschitz
homomorphism. For proofs and more detailed discussion the reader is
referred to~\cite{Matveev-Asymptotic-2018}.

\subsection{Tropical diagrams}
Applying the construction of the previous section we obtain its
tropicalization -- a semi-module $(\prob[\Gbf],+,\,\cdot\,,\aikd)$.
The restriction of the asymptotic distance on the original monoid can
be defined independently as
\[
\aikd(\Xcal,\Ycal):=\lim_{n\to\infty}\frac1n\ikd(\Xcal^{n},\Ycal^{n})
\]

One of the main tools for the estimation of the (asymptotic) distance
is the so-called Slicing Lemma and its following consequence.
\begin{proposition}{p:slicingcorollary}
  Let $\Gbf$ be an indexing category, $\Xcal,\Ycal\in\prob\<\Gbf\>$ and
  $U\in\prob$.
  \begin{enumerate}
  \item
    \label{p:slicingreduction}      
    Let $\Xcal\to U$ be a reduction, then
    \begin{align*}
      \ikd(\Xcal,\Ycal) 
      &\leq 
      \int_{U}\ikd(\Xcal\rel u,\Ycal)\d p_{U}(u)+
      \size{\Gbf}\cdot\ent(U)
    \end{align*}
  \item
    \label{p:slicingcofan}
    For a ``co-fan'' $\Xcal\to U\ot\Ycal$ holds
    \[
    \ikd(\Xcal,\Ycal)
    \leq
    \int_{U}\ikd(\Xcal\rel u,\Ycal\rel u)\d p_{U}(u)
    \]
  \end{enumerate}
\end{proposition}

The statements and the proofs of the Slicing Lemma and its consequences can
be found in~\cite{Matveev-Asymptotic-2018}.

We will show below that $(\prob\langle \Gbf \rangle,\otimes,\aikd)$ has the uniformly bounded and vanishing defect properties. For this purpose we
need to develop some technical tools.

\subsection{Mixtures}
The input data for the mixture operation is
a family of $\Gbf$-diagrams, parameterized by a probability space. As
a result one obtains another $\Gbf$-diagram with pre-specified
conditionals. One particular instance of a mixture is when one mixes
two diagrams $\Xcal$ and $\set{\bullet}^{\Gbf}$, the latter being a
constant $\Gbf$-diagram of one-point probability spaces. This
operation will be used as a substitute for taking radicals
``$\Xcal^{\frac1n}$'' below.

\subsubsection{Definition of mixtures}
Let $\Gbf$ be an indexing category and $\Theta$ be a
probability space. By $\Theta^{\Gbf}$ we denote the
  \term{constant $\Gbf$-diagram} -- the diagram such that all
    spaces in it are $\Theta$ and all morphisms are identity
    morphisms. Let $\set{\Xcal_{\theta}}_{\theta\in\un\Theta}$ be
a family of $\Gbf$-diagrams parameterized by $\Theta$. The
\term{mixture} of the family $\set{\Xcal_{\theta}}$ is the reduction
\[      
  \mix\set{\Xcal_{\theta}}=
  \left(
    \Ycal
    \too
    \Theta^{\Gbf}
  \right)
\]      
such that
\[\tageq{mixture}
  \Ycal\rel\theta\cong\Xcal_{\theta}
  \quad\text{for any $\theta\in\Theta$}
\]

The mixture exists and is uniquely defined by
property~(\ref{eq:mixture}) up to an isomorphism which is identity on
$\Theta^{\Gbf}$.

We denote the top diagram of the mixture by
\[
  \Ycal=:\bigoplus_{\theta\in\Theta}\Xcal_{\theta}
\]
and also call it the mixture of the family $\set{\Xcal_{\theta}}$.

When 
\[
  \Theta=\Lambda_{\alpha}
  :=
  \big(\set{\square,\blacksquare};
  p(\blacksquare)=\alpha
  \big) 
\]
is a binary space we write simply
\[
  \Xcal_{\blacksquare}\oplus_{\Lambda_{\alpha}}\Xcal_{\square}
\]
for the mixture. The diagram subindexed by the $\blacksquare$
will always be the first summand.

The entropy of the mixture can be evaluated by the following formula
\[
  \ent_{*}\left(\bigoplus_{\theta\in\Theta}\Xcal_{\theta}\right)
  =
  \int_{\Theta}\ent_{*}(\Xcal_{\theta})\d
  p(\theta) + \ent_{*}(\Theta^{\Gbf})
\]
Mixtures satisfy the distributive law with respect to the tensor
product
\begin{align*}
  \mix(\set{\Xcal_{\theta}}_{\theta\in\Theta})
  \otimes
  \mix(\set{\Ycal_{\theta'}}_{\theta'\in\Theta'})
  &\cong
  \mix(\set{\Xcal_{\theta}\otimes\Ycal_{\theta'}}
  _{(\theta,\theta')\in\Theta\otimes\Theta'})
  \\
  \left(\bigoplus_{\theta\in\Theta}\Xcal_{\theta}\right)
  \otimes
  \left(\bigoplus_{\theta'\in\Theta'}\Ycal_{\theta'}\right)
  &\cong
  \bigoplus_{(\theta,\theta')\in\Theta\otimes\Theta'}
      (\Xcal_{\theta}\otimes\Ycal_{\theta'})
\end{align*}
  
\subsubsection{The distance estimates for the mixtures.}
Recall that for a diagram category $\Gbf$ we denote by
$\set{\bullet}=\set{\bullet}^{\Gbf}$ the constant $\Gbf$-diagram of
one-point spaces.
  
The mixture of a $\Gbf$-diagram with $\set{\bullet}^{\Gbf}$ may serve
as an substitute of taking radicals of the diagram.  The following lemma
provides a justification of this by some distance estimates related to
mixtures and will be used below.

\begin{lemma}{p:mixtures-dist1}
  Let $\Gbf$ be a complete diagram category and
  $\Xcal,\Ycal\in\prob\<\Gbf\>$. Then
  \begin{enumerate}
  \item
    $\displaystyle
    \aikd(\Xcal,\Xcal^{n}\oplus_{\Lambda_{1/n}}\set{\bullet})
    \leq
    \ent(\Lambda_{1/n})
    $
  \item
    $\displaystyle
    \aikd\big(\Xcal,(\Xcal\oplus_{\Lambda_{1/n}}\set{\bullet})^{n}\big)
    \leq
    n\cdot\ent(\Lambda_{1/n})
    $
  \item
    $\displaystyle
    \aikd\big(
    (\Xcal\otimes\Ycal)\oplus_{\Lambda_{1/n}}\set{\bullet},
    (\Xcal\oplus_{\Lambda_{1/n}}\set{\bullet})
    \otimes
    (\Ycal\oplus_{\Lambda_{1/n}}\set{\bullet})
    \big)
    \leq
    3\ent(\Lambda_{1/n})
    $
  \item
    $\displaystyle
    \aikd\big((\Xcal\oplus_{\Lambda_{1/n}}\set{\bullet}),
    (\Ycal\oplus_{\Lambda_{1/n}}\set{\bullet})\big)
    \leq
    \frac1n\aikd(\Xcal,\Ycal)
    $
  \end{enumerate}
\end{lemma}
Note that the distance estimates in the lemma above are with respect
to the asymptotic distance. This is essential, since from
the perspective of the intrinsic distance mixtures are very
badly behaved.
\begin{proof}
  For $\lambda \in \Lambda_{1/n}^N$, define $\emp(\lambda)$ to be the number of black squares in the sequence $\lambda$.
  It is a binomially
  distributed random variable with mean $N/n$ and variance
  $\frac{N}{n}(1-\frac1n)$.
  
  The first claim is then proven by the following calculation
  \begin{align*}
    \aikd(\Xcal&,\Xcal^{n}\oplus_{\Lambda_{1/n}}\set{\bullet})
    =
    \lim_{N\to\infty}
      \frac1N
        \ikd\left(
          \Xcal^{N},
          (\Xcal^{n}\oplus_{\Lambda_{1/n}}\set{\bullet})^{N}
        \right)
    \\    
    &=
    \lim_{N\to\infty}
      \frac1N
        \ikd\left(
          \Xcal^{N},
          \bigoplus_{\lambda\in\Lambda_{1/n}^{N}}
              \Xcal^{n\cdot \emp(\lambda)}
        \right)
    \\
    &\leq
    \ent(\Lambda_{1/n}) +
    \lim_{N\to\infty}
      \frac1N
      \int_{\lambda\in\Lambda^{n}_{1/n}}
        \ikd(\Xcal^{N},
             \Xcal^{n\cdot \emp(\lambda)})
      \d p(\lambda)
    \\
    &\leq
    \ent(\Lambda_{1/n}) +
    \|\ent_{*}(\Xcal)\|_{1}\cdot
    \lim_{N\to\infty}
      \frac{n}{N}
      \cdot
      \int_{\lambda\in\Lambda_{1/n}^{N}}
         \big|N/n- \emp(\lambda)\big|
         \d p(\lambda)
    \\
    &\leq
    \ent(\Lambda_{1/n}) +
    \|\ent_{*}(\Xcal)\|_{1}\cdot
    \lim_{N\to\infty}
      \frac{n}{N}\cdot\sqrt{N\cdot\frac1n(1-\frac1n)}
    =
    \ent(\Lambda_{1/n})
  \end{align*}
   where we used Proposition~\ref{p:slicingcorollary}(\ref{p:slicingreduction}) for the
  inequality on the third line above.

  The second claim is proven similarly and the third follows from the
  second and the $1$-Lipschitz property of the tensor
  product:
  \begin{align*}
  &\aikd\big(
  (\Xcal\otimes\Ycal)\oplus_{\Lambda_{1/n}}\set{\bullet},
  (\Xcal\oplus_{\Lambda_{1/n}}\set{\bullet})
  \otimes
  (\Ycal\oplus_{\Lambda_{1/n}}\set{\bullet})
  \big)
   \\
  &\leq \aikd\big(
  (\Xcal\otimes\Ycal)\oplus_{\Lambda_{1/n}}\set{\bullet},
  \Xcal
  \otimes
  \Ycal
  \big) + 2 \ent(\Lambda_{1/n}) \\
  &\leq
  3\ent(\Lambda_{1/n})
  \end{align*}
  Finally, the fourth follows from
  Proposition~\ref{p:slicingcorollary}(\ref{p:slicingcofan}), by slicing
  both arguments along $\Lambda_{1/n}$.
\end{proof}

\subsection{Vanishing defect property and completeness of the
  tropical cone}
\begin{lemma}{p:unifsmalldefectaikd}
  For every admissible function $\phi$, every
  $\bar\Xcal\in\qlin_{\phi}(\prob\<\Gbf\>,\aikd)$ and every $k \in
  \Nbb$, there exists an asymptotically equivalent sequence
  $\bar\Ycal$ with defect bounded by the admissible function $\phi_k$ defined by
  \[
    \phi_k(s) := 3\ent(\Lambda_{1/k}) + \frac{1}{k} \phi(k \cdot s)
  \] 
\end{lemma} 
\begin{proof}
  Let $\bar\Xcal=\set{\Xcal(i)}$ be a quasi-linear sequence with
  defect bounded by $\phi$ and let $k \in \Nbb$.
  
  Define a new sequence $\bar\Ycal=\set{\Ycal(i)}$ by
  \[
    \Ycal(i)
    :=
    \big(\Xcal(k\cdot i)\big)
    \oplus_{\Lambda_{1/k}}
    \set{\bullet}
  \]
  First we verify
  that the sequences $\bar\Xcal$ and $\bar\Ycal$ are asymptotically
  equivalent, that is
  \begin{align*}
    \hat\aikd(\bar\Xcal,\bar\Ycal)
    &:=
    \lim_{i\to\infty}
    \frac{1}{i}
    \aikd\left(\Xcal(i), 
    \Ycal(i)
    \right)
    =
    0
  \end{align*}
  We estimate the asymptotic distance between individual members of
  sequences $\bar\Xcal$ and $\bar\Ycal$ using
  Lemma~\ref{p:mixtures-dist1} and Corollary~\ref{p:psi-homo} as
  follows
  \begin{align*}
    \aikd(&\Xcal(i), 
    \Ycal(i)
    )
    =
    \aikd\big(\Xcal(i), 
    \Xcal(k\cdot i)
    \oplus_{\Lambda_{1/k}}\set{\bullet}
    \big)    
    \\
    &\leq
    \aikd\left(\Xcal(i), 
    \Xcal(i)^{k}
    \oplus_{\Lambda_{1/k}}\set{\bullet}
    \right)
    +    
    \aikd\left(\Xcal(i)^{k}
    \oplus_{\Lambda_{1/k}}\set{\bullet}, 
    \Xcal(k\cdot i)
    \oplus_{\Lambda_{1/k}}\set{\bullet}
    \right)    
    \\
    &\leq
    \ent(\Lambda_{1/k})
    +    
    D_{\phi}\cdot \phi(i)
  \end{align*}
  Thus $\hat\aikd(\bar\Xcal,\bar\Ycal)=0$ and the two sequences are
  asymptotically equivalent.  Next we show that the sequence
  $\bar\Ycal$ is $\aikd$-quasi-linear and evaluate its defect, also using
  Lemma~\ref{p:mixtures-dist1}. Let $i,j\in\Nbb$, then
  \begin{align*}
    \aikd&\big(\Ycal(i+j),\Ycal(i)\otimes\Ycal(j)\big)
    \\
    &=
    \aikd
    \Big(
    \Xcal(k\cdot i+k\cdot j)
    \oplus_{\Lambda_{1/k}}\set{\bullet},
    \big(
      \Xcal(k\cdot i)
      \oplus_{\Lambda_{1/k}}\set{\bullet}
    \big)
    \otimes
    \big(
      \Xcal(k\cdot j)
      \oplus_{\Lambda_{1/k}}\set{\bullet}
    \big)
    \Big)    
    \\
    &\leq
    \aikd\Big(\big(\Xcal(k\cdot i)\!\otimes\!\Xcal(k\cdot j)\big)
    \!\oplus_{\Lambda_{1/k}}\!\set{\bullet},     
    \big(
      \Xcal(k\cdot i)
      \oplus_{\Lambda_{1/k}}\!\set{\bullet}
      \big)
    \!\otimes\!
    \big(
      \Xcal(k\cdot j)
      \oplus_{\Lambda_{1/k}}\!\set{\bullet}
      \big)
    \Big)    
    \\
    &\quad+
    \frac{1}{k}\phi\big(k \cdot (i + j )\big)
    \\
    &\leq
    3\ent(\Lambda_{1/k})+ \frac{1}{k} \phi\big(k \cdot( i + j)\big)
  \end{align*}
\end{proof}

\begin{corollary}{p:vanishingdefect}
  For any indexing category $\Gbf$ and for the admissible function
  $\phi$ given by $\phi(t) = t^{\alpha}$, $\alpha \in [0, 1)$,
    $\qlin_{\phi}(\prob\<\Gbf\>,\aikd)$ has the uniformly bounded and
    vanishing defect properties.
\end{corollary}
\begin{proof}
  Let $\bar\Xcal\in\qlin_{\phi}(\prob\<\Gbf\>,\aikd)$.  By Lemma
  \ref{p:unifsmalldefectaikd} there exists an asymptotically
  equivalent sequence $\bar\Ycal$ with defect bounded by $\phi_k$
  defined by
  \begin{align*}
  \phi_k(t) &:= 3\ent(\Lambda_{1/k}) + \frac{1}{k} C \phi(k \cdot t) \\
  &=  3\ent(\Lambda_{1/k}) + \frac{1}{k} C (k \cdot t)^\alpha
  \end{align*}
  Hence there exists a sequence $c_k \to 0$ such that for all $t \geq 1$,
  \[
  \phi_k(t) \leq c_k t^\alpha
  \]
  showing the uniformly bounded and vanishing defect property.
\end{proof}

\subsection{Diagrams of tropical probability spaces}
By applying the general setup in the previous section to the metric
Abelian monoids $(\prob\<\Gbf\>, \otimes, \ikd)$ and $(\prob\<\Gbf\>,
\otimes, \aikd)$ and using the Corollary~\ref{p:vanishingdefect} we
obtain the following theorem.

\begin{theorem}{p:corollary-acone-ikd-aikd}
  Fix an admissible function $\phi$ and consider the commutative
  diagram
  \[\tageq{tropical-diagram}
  \begin{cd}[row sep=tiny]
    \mbox{}
    \&
    \big(\lin(\prob\<\Gbf\>,\ikd) , \aikd \big)
    \arrow[hookrightarrow]{dd}{\i_{1}}
    \arrow[hookrightarrow]{r}{\j_{1}}
    \&
    \big(\qlin_{\phi}(\prob\<\Gbf\>,\ikd) , \aikd\big)
    \arrow[hookrightarrow]{dd}{\i_{2}}
    \\
    (\prob\<\Gbf\>,\aikd)
    \arrow{ru}{f}
    \arrow{rd}{f'}
    \\
    \mbox{}
    \&
    \big(\lin(\prob\<\Gbf\>,\aikd), \hat\aikd \big)
    \arrow[hookrightarrow]{r}{\j_{2}}
    \&
    \big(\qlin_{\phi}(\prob\<\Gbf\>,\aikd), \hat\aikd \big)
  \end{cd}
  \]
  Then the following statements hold:
  \begin{enumerate}
  \item 
    The maps $f, f', \i_1$ are isometries.
  \item 
    The maps $\i_2, \j_1,\j_2$ are isometric embeddings and each map
    has a dense image in the corresponding target space.
  \item 
    The space in the lower-right corner,
    $\big(\qlin_{\phi}(\prob\<\Gbf\>,\aikd), \hat\aikd \big)$, is
    complete.
  \end{enumerate}
\end{theorem} 
We would like to conjecture that all maps in the diagram above are
isometries.

Since  $\qlin_{\phi}(\prob\<\Gbf\>,\aikd)$ is complete and
has $\lin(\prob\<\Gbf\>,\aikd)$ as a dense subset for any $\phi>0$,
it follows that
$\qlin_{\phi}(\prob\<\Gbf\>,\aikd)$ does not depend (up to isometry of
pseudo-metric spaces) on the choice of
admissible $\phi>0$. From now on  we will choose the particular function
$\phi(t):=t^{3/4}$. The choice
will be clear when we formulate the Asymptotic Equipartition Property
for diagrams.
We may finally define the space of \term{tropical
  $\Gbf$-diagrams}, as the space in the lower-right corner of the
diagram
\[
  \prob[\Gbf]
  :=
  \big(\qlin_{\phi}(\prob\<\Gbf\>, \aikd),\otimes,\cdot,\hat\aikd\big)
\]
By Theorem~\ref{p:corollary-acone-ikd-aikd} above, this space is
complete.
  
The entropy function $\ent_{*}:\prob\<\Gbf\>\to\Rbb^{\Gbf}$
extends to a linear functional
\[
  \ent_{*}:\prob[\Gbf]\to(\Rbb^{\Gbf},\|\cdot\|_{1})
\]
of norm one, defined by
\[
  \ent_{*}(\bar\Xcal) = \lim_{n\to \infty } \frac1n \ent_{*} \big(\Xcal(n)\big)
\]

\section{AEP}
\label{se:aep}
\subsection{Homogeneous diagrams}
A $\Gbf$-diagram $\Xcal$ is called \term{homogeneous} if the
automorphism group $\Aut(\Xcal)$ acts transitively on every space in
$\Xcal$. Homogeneous probability spaces are uniform. For more complex
indexing categories this simple description is not sufficient.
The subcategory of all homogeneous $\Gbf$-diagrams will
be denoted $\Prob\<\Gbf\>_{\hsf}$. This space is invariant under the
tensor product, thus it is a metric Abelian monoid. 

\subsubsection{Universal construction of homogeneous diagrams} 
Examples of homogeneous dia\-grams could be constructed in the
following manner. Fix a finite group $G$ and consider a $\Gbf$-diagram
$\set{H_{i};\alpha_{ij}}_{i\in\Gbf}$ of subgroups of $G$, where
morphisms $\alpha_{ij}$ are inclusions. The $\Gbf$-diagram of
probability spaces $\set{X_{i};f_{ij}}$ is constructed by setting
$X_{i}=(G/H_{i},\unif)$ and taking $f_{ij}$ to be the natural
projection $G/H_{i}\to G/H_{j}$, whenever $H_{i}\subset H_{j}$.  The
resulting diagram $\Xcal$ will be minimal if and only if for any
$i,j\in\Gbf$ there is $k\in\Gbf$, such that $H_{k}=H_{i}\cap H_{j}$.
In fact, any homogeneous diagram arises this way, see
\cite{Matveev-Asymptotic-2018}.

\subsection{Asymptotic Equipartition Property}
In~\cite{Matveev-Asymptotic-2018} the following theorem is proven.
\begin{theorem}{p:aep-complete}
  Suppose $\Xcal\in\prob\<\Gbf\>$ is a $\Gbf$-diagram of
  probability spaces for some fixed indexing category $\Gbf$.
  Then there exists a sequence
  $\bar\Hcal=(\Hcal_{n})_{n=0}^{\infty}$ of homogeneous
  $\Gbf$-diagrams such
  that
  \[\tageq{quantaep} 
    \frac{1}{n} 
    \ikd (\Xcal^{\otimes n},\Hcal_{n}) 
    \leq 
    C(|X_0|,\size{\Gbf}) \cdot \sqrt{\frac{\ln^3 n}{n}} 
  \]
where $C(|X_0|, \size{\Gbf})$ is a constant only depending on $|X_0|$
and $\size{\Gbf}$.
\end{theorem}

Defining
\[
\prob[\Gbf]_{\hsf}:=\qlin_{\phi}(\prob\<\Gbf\>_{\hsf},\aikd)
\]
the Asymptotic Equipartition Property can be reformulated as in the
Theorem~\ref{p:aep-tropical} below. 
\begin{theorem}{p:aep-tropical}
  For any indexing category $\Gbf$ the image of the natural inclusion
  \[
    \prob[\Gbf]_{\hsf}\hookrightarrow\prob[\Gbf]
  \]
  is dense.
\end{theorem}
\proof By Theorem \ref{p:aep-complete}, every linear sequence can be
approximated by a homogeneous sequence. It follows from the bound
(\ref{eq:quantaep}) that the defect of the approximating homogeneous
sequence is bounded by a constant times $\phi$, defined by
  $\phi(t)=t^{3/4}$. Moreover, the linear sequences are dense by
Theorem \ref{p:corollary-acone-ikd-aikd}. This finishes the proof.
\eproof

\section{The tropical cone for probability spaces and chains}
\label{se:chains}
Although for general indexing categories $\Gbf$ the space of tropical
$\Gbf$-diagrams is infinite dimensional, it has a very simple,
finite-dimensional description if $\Gbf$ consists of a single object,
or if it is a special type of indexing categories called a \term{chain}.

The chain of length $k$, denoted by $\Cbf_k$, is the indexing category
with $k$ objects $O_1, \dots, O_k$, and a morphism from $O_i$ to $O_j$
whenever $i \geq j$. A $\Cbf_k$-diagram of probability spaces is then
a chain of reductions
\[
  X_k \to X_{k-1} \to \cdots \to X_{1}
\]
Recall that homogeneous probability spaces are (isomorphic to)
probability spaces with a uniform distributions.  \emph{Homogeneous}
chains have a very simple description as well. A chain $\Hcal \in
\prob\<\Cbf_k\>$ is homogeneous if and only if the individual
probability spaces are homogeneous, i.e.~if and only if the individual
probability spaces are (isomorphic to) probability spaces with a
uniform measure.

Based on this simple description we derive the following theorem.

\begin{theorem}{p:spaces-chains}
  For $k \in \Nbb$, the tropical cone $\prob[\Cbf_k]$ is isomorphic to
  the following cone in $(\Rbb^k, |\cdot|_1)$:
  \[
    \set{
      \left(
      \begin{array}{c}
	x_k \\
	\vdots \\
	x_1
      \end{array}
      \right)	
      \in 
      \Rbb^k \
      \middle|\
      0 \leq x_1 \leq \dots \leq x_k
    }
  \]
  In particular, the algebraic structure and the pseudo-distance are
  preserved under the isomorphism.
\end{theorem}

In case of single probability spaces, Theorem \ref{p:spaces-chains} is
a direct consequence of the asymptotic equipartition property and the
following lemma. For chains, a similar argument works.

\begin{lemma}{p:aikd-uniform}
  Denote by $U_n$ a finite uniform probability space of
  cardinality $n$, then \[ \tageq{ikd-uniform} \ikd(U_n,
  U_m) \leq 2 \ln 2 + \left| \ln \frac{n}{m} \right| \]
  and \[ \tageq{aikd-uniform} \aikd(U_n, U_m) = \left| \ent(U_n)
  - \ent(U_m) \right| \]
\end{lemma}
\begin{proof}
  We will construct a specific two-fan $U_n \ot[f] U_{nm} \to[g] U_m$.
  Identify $U_\ell$ with $\{0,\dots,\ell-1\}$. Let $k \in
  U_{nm}$. Then $k$ can be written uniquely as
  \[
  \begin{cases}
    k = i_0 \cdot m + j_0 &\text{with } i_0 \in U_n, \quad j_0 \in U_m\\
    k = i_1 \cdot n + j_1 &\text{with } i_1 \in U_m, \quad j_1 \in U_n
  \end{cases}
  \]
  and we set $f(k) := i_0$ and $g(k) := i_1$.
  
  Now that we have constructed a two-fan $U_n \ot[f] U_{nm} \to[g] U_m$,
  let $U_n \ot Z \to U_m$ be its minimal reduction. We estimate
  $|Z| \leq n+ m$, which implies that
  \[
  \begin{split}
    \ikd(U_n, U_m)
    &\leq
    2 \ent(Z) - \ent(U_n) - \ent(U_m)\\
    &\leq 2 \ln(n+m) - \ln n - \ln m\\
    &\leq 2 \ln 2 + 2 \ln \max \{n, m\} - \ln n - \ln m \\
    &\leq 2 \ln 2 + \left| \ln \frac{n}{m} \right|
  \end{split}
  \]
  thus establishing inequality (\ref{eq:ikd-uniform}).

  To show equality (\ref{eq:aikd-uniform}), recall that the entropy as
  a map is $\ikd$-Lipschitz with Lipschitz constant $1$. Therefore, we
  have
  \[
    | \ent(U_n) - \ent(U_m) | 
    \leq 
    \ikd(U_n, U_m) \leq | \ent(U_n) - \ent(U_m)| + 2 \ln 2
  \]
  and
  \[
  \begin{split}
    \aikd(U_n, U_m)
    &= 
    \lim_{\ell \to \infty} \frac{1}{\ell} \ikd\big(U_n^\ell,
        U_m^\ell\big) 
    = 
    | \ent(U_n) - \ent(U_m)  |
  \end{split}
  \]
\end{proof}

\bibliographystyle{alpha}       
\bibliography{ReferencesAC}
\end{document}